\documentclass[12pt]{article}
\usepackage[UKenglish]{babel}
\usepackage[T1]{fontenc}
\usepackage{lmodern,amsmath,amsthm,amsfonts,amssymb,graphicx,float,microtype,thmtools,underscore,mathtools,thm-restate}
\usepackage[shortlabels]{enumitem}
\setlist[itemize]{topsep=0ex,itemsep=0ex,parsep=0ex}
\setlist[enumerate]{topsep=0ex,itemsep=0ex,parsep=0ex}
\usepackage[usenames,dvipsnames,svgnames,table]{xcolor}
\usepackage{todonotes}
\usepackage[unicode=true]{hyperref}
\usepackage{breakurl}
\hypersetup{ 
colorlinks,
linkcolor={blue!60!black},
citecolor={black},
urlcolor={blue!60!black},
pdftitle={An Induced A-Path Theorem}}
\usepackage[capitalise, compress, nameinlink, noabbrev]{cleveref}
\crefname{lem}{Lemma}{Lemmas}
\crefname{thm}{Theorem}{Theorems}
\crefname{ques}{Question}{Theorems}
\crefname{cor}{Corollary}{Corollaries}
\crefname{enumi}{Item}{Items}
\crefformat{equation}{(#2#1#3)}
\Crefformat{equation}{Equation #2(#1)#3}
\crefformat{enumi}{#2#1#3}
\Crefformat{enumi}{Item (#2#1#3)}

\newcommand{\defn}[1]{\textcolor{Maroon}{\emph{#1}}}
\usepackage[longnamesfirst,numbers,sort&compress]{natbib}
\makeatletter
\def\NAT@spacechar{~}
\makeatother
\usepackage[margin=28mm]{geometry}
\renewcommand{\baselinestretch}{1.1}
\allowdisplaybreaks

\DeclarePairedDelimiter{\floor}{\lfloor}{\rfloor}

\renewcommand{\epsilon}{\varepsilon}
\renewcommand{\emptyset}{\varnothing}
\renewcommand{\geq}{\geqslant}
\renewcommand{\leq}{\leqslant}

\DeclareMathOperator{\dist}{dist}

\newcommand{\A}{\mathcal{A}}

\newcommand{\NN}{\mathbb{N}}

\renewcommand{\thefootnote}{\fnsymbol{footnote}}
\theoremstyle{plain}
\newtheorem{thm}{Theorem}
\newtheorem{lem}[thm]{Lemma}

\newtheorem{obs}[thm]{Observation}
\newtheorem{claim}{Claim}
\crefname{obs}{Observation}{Observations}
\newtheorem*{lem*}{Lemma}
\theoremstyle{definition}
\newtheorem{conj}[thm]{Conjecture}
\newtheorem*{conj*}{Conjecture}

\presetkeys%
{todonotes}%
{inline}{}
\date{}

\begin{document}

\title{\bf\fontsize{18pt}{18pt}\selectfont An Induced $A$-Path Theorem}

\author{Robert Hickingbotham\,\footnotemark[1]	\quad \quad \quad Gwena\"el Joret\,\footnotemark[1]	}

\footnotetext[1]{D\'epartement d'Informatique, Universit\'e libre de Bruxelles, Belgium ({\tt robert.hickingbotham@ulb.be}, {\tt gwenael.joret@ulb.be}). R.\ Hickingbotham and G.\ Joret are supported by the Belgian National Fund for Scientific Research (FNRS).}

\maketitle

\begin{abstract}
     Given a graph $G$ and $\A\subseteq V(G)$, a classical theorem of Gallai (1964) states that for every positive integer $k$, the graph $G$ contains $k$ pairwise vertex-disjoint $\A$-paths, or a set $Z\subseteq V(G)$ of size at most $2(k-1)$ such that $G-Z$ contains no $\A$-paths. We generalise Gallai's theorem to the induced setting: We prove that $G$ contains $k$ pairwise anti-complete $\A$-paths, or a set $Z$ of size at most $54(k-1)$ such that, after removing the closed neighbourhood of $Z$, the resulting graph has no $\A$-path. Here, two paths are anti-complete if they are vertex disjoint and there is no edge in $G$ having one endpoint in each of them.  
    
    We further show that the bound $54(k-1)$ on the size of $Z$ can be reduced to $4(k-1)$ if one removes the balls of radius $4$ around the vertices of $Z$ (instead of radius $1$), which is within a factor $2$ of optimal. We also establish analogous results for long induced $\A$-paths. 
\end{abstract}

\renewcommand{\thefootnote}{\arabic{footnote}}

\section{Introduction}

Let $G$ be a graph and $\A\subseteq V(G)$. An \defn{$\A$-path} in $G$ is a path with at least one edge whose endpoints belong to $\A$. The classic $\A$-path theorem of \citet{Gallai1964} states that for every $k\in \NN$, $G$ contains $k$ pairwise vertex-disjoint $\A$-paths, or there exists a set $Z\subseteq V(G)$ with size at most $2k-2$ such that $G-Z$ contains no $\A$-paths. This min-max theorem was a key ingredient in the original proof of the celebrated Erd{\H{o}}s-P{\'o}sa theorem~\cite{ErdosPosa1965}: for every $k\in \NN$, every graph $G$ contains $k$ vertex disjoint cycles or a set $X$ with $|X|\in \mathcal{O}(k\log(k))$ such that $X$ intersects all cycles in $G$. This result initiated a long line of research into other so-called Erd{\H{o}}s-P{\'o}sa type theorems; see \cite{RT17} for a survey.

In recent years, attention has turned to coarse Erd{\H{o}}s-P{\'o}sa type theorems, where `disjoint' is replaced with `far apart' \cite{AGHK2025EP,DJMM2024EP}. This direction is motivated by the emerging field of coarse graph theory, which seeks to understand the large-scale geometry of graphs. One of the central questions in this area is the Coarse Menger Conjecture, independently posed by \citet{Albrechtsen2024Menger} and \citet{georgakopoulos2023graph}. For a graph $G$, $X\subseteq V(G)$, and $r\in \NN$, let $N_G[X,r]:=\{v\in V(G)\colon \dist_G(X,v)\leq r\}$ and $N_G[X]:=N_G[X,1]$.

\begin{conj}[\cite{Albrechtsen2024Menger,georgakopoulos2023graph}]\label{CoarseMenger}
    For all $k,d\in \NN$, there exists $c\in \NN$ such that every graph $G$ and $X,Y\subseteq V(G)$ contains one of the following:
   \begin{enumerate}
        \item[(i)] $k$ $(X,Y)$-paths $P_1,\dots,P_k$ such that $\dist_G(P_i,P_j)\geq d$ for all distinct $i,j\in [k]$; or
        \item[(ii)] a set $Z\subseteq V(G)$ with $|Z|\leq k-1$ such that $G-N_G[Z,cd]$ has no $(X,Y)$-path.
    \end{enumerate} 
\end{conj}

Unfortunately, \citet{NSS2024counterexample} showed that \cref{CoarseMenger} is false for all $d\geq 3$. Furthermore, the same set of authors more recently showed that it remains false even when the number of balls is allowed to depend on a function of $k$ \cite{NSS2025AsymptoticIV}. See \cite{Albrechtsen2024Menger,georgakopoulos2023graph,hendrey2023induced,gartland2023induced,NSS2025CoarseMengerPW,NSS2025DistantPaths} for positive resolution of special cases of \cref{CoarseMenger}.

In many settings, when a Menger-type theorem fails, the corresponding $\A$-path version holds. This phenomenon has been particularly fruitful in the study of group-labelled graphs \cite{CGGGLS2006Apaths,BHJ2018Frames,Bruhn2022Packing,Thomas2023Packing, Wollan2010Packing}. Motivated by this, Geelen proposed a Coarse Gallai Conjecture.\footnote{Posed by Geelen at the Barbados Graph Theory Workshop in March 2024 held at the Bellairs Research Institute of McGill University in Holetown.}

\begin{conj}[Geelen 2024]\label{GeelenConj}
    There exist functions $f,g$ such that, for all $k,d\in \NN$, every graph $G$ and every vertex set $\A\subseteq V(G)$ satisfies the following:
     \begin{enumerate}
        \item[(i)] $k$ $\A$-paths $P_1,\dots,P_k$ such that $\dist_G(P_i,P_j)\geq d$ for all distinct $i,j\in [k]$; or
        \item[(ii)] a set $Z\subseteq V(G)$ with $|Z|\leq f(k)$ such that $G-N_G[Z,g(d)]$ has no $\A$-path.
    \end{enumerate} 
\end{conj}

This paper resolves the $d=2$ case of Geelen's conjecture. For a graph $G$, two sets $X,Y\subseteq V(G)$ are \defn{anti-complete} if $X\cap Y=\emptyset$ and there is no edge in $G$ with one end in $X$, the other in $Y$.

\begin{thm}\label{MainApath}
    For every $k\in \NN$, every graph $G$ and $\A\subseteq V(G)$, $G$ contains $k$ pairwise anti-complete $\A$-paths, or there exists a set $Z\subseteq V(G)$ with $|Z|\leq 54(k-1)$ such that $G-N_G[Z]$ has no $\A$-paths.
\end{thm}

\Cref{MainApath} marks a positive development for coarse graph theory against the backdrop of the recent trend of counterexamples in the area \cite{NSS2024counterexample, NSS2025AsymptoticIV, DHIM2024fat, AD2025counterexample,ADG2026Small}. The use of radius-$1$ balls is best possible: the graph $K_n$ with $\A=V(K_n)$ contains neither $2$ pairwise anti-complete $\A$-paths nor a set $Z\subseteq V(K_n)$ with $|Z|\leq n-2$ such that $K_n-Z$ has no $\A$-path. 

In addition to \cref{MainApath}, we prove the following induced $\A$-path theorem, which reduces the number of balls to nearly optimal at the expense of larger ball radii.

\begin{thm}\label{MainApathLargeBalls}
    For every $k\in \NN$, every graph $G$ and $\A\subseteq V(G)$, $G$ contains $k$ pairwise anti-complete $\A$-paths, or there exists a set $Z\subseteq V(G)$ with $|Z|\leq 4(k-1)$ such that $G-N_G[Z,4]$ has no $\A$-paths.
\end{thm}

The number of balls in \cref{MainApathLargeBalls} is within a factor of two from optimal. To see this, let $r\in \NN$ and $G$ be the graph obtained from the complete graph on $2k-1$ vertices by replacing each edge with a path of length $3r$. Let $\A$ be the set of high-degree vertices of $G$. Then $G$ neither contains $k$ pairwise anti-complete $\A$-paths, nor a set $Z\subseteq V(G)$ of size at most $2k-3$ such that $G-N_G[Z,r]$ has no $\A$-path.

We further extend \cref{MainApath,MainApathLargeBalls} to long induced $\A$-paths. For a graph $G$, an \defn{induced path} is a path that is an induced subgraph of $G$. The \defn{length} of a path is the number of edges it contains.

\begin{restatable}{thm}{LongInducedAPaths}\label{LongInducedAPaths}
		 For all $k,\ell\in \NN$, every graph $G$ and $\A\subseteq V(G)$ contains one of the following:
    \begin{enumerate}
        \item[(i)] $k$ pairwise anti-complete induced $\A$-paths of length at least $\ell$; or
        \item[(ii)] sets $Z_1,Z_2\subseteq V(G)$ with $|Z_1|\leq (12\max\{\ell,3\}+18)(k-1)$ and $|Z_2|\leq 4(k-1)$ such that neither $G-N_G[Z_1]$ nor $G-N_G[Z_2,\max\{\ell+1,4\}]$ contains an induced $\A$-path of length at least $\ell$.
    \end{enumerate} 
\end{restatable}

Setting $\ell=1$ recovers \cref{MainApath,MainApathLargeBalls}.

Beyond coarse graph theory, our motivation for studying induced $\A$-paths stems from the study of induced minors. A graph $H$ is an \defn{induced minor} of a graph $G$ if it can be obtained from an induced subgraph of $G$ by contracting edges. In recent years, considerable progress has been made in understanding the global structure of $H$-induced-minor-free graphs when $H$ is planar, yet comparatively little is known when $H$ is non-planar. We hope that \cref{MainApath} may prove useful in establishing an `Induced Minor Structure Theorem', one that is akin to the celebrated Graph Minor Structure Theorem~\cite{robertson2003graph}. See \cite{Hickingbotham2025InducedMinors} for an application of $\A$-paths to induced minors.

\cref{MainApath} also provides evidence for the following Induced Menger Conjecture, which corresponds to the $d=2$ case of \cref{CoarseMenger}.

\begin{conj}\label{InducedMenger}
    For every $k\in \NN$, every graph $G$ and $X,Y\subseteq V(G)$ satisfies one of the following:
   \begin{enumerate}
        \item[(i)] $G$ contains $k$ pairwise anti-complete $(X,Y)$-paths; or
        \item[(ii)] there exists a set $Z\subseteq V(G)$ with $|Z|\leq k$ such that $G-N_G[Z]$ has no $(X,Y)$-path.
    \end{enumerate} 
\end{conj}
 
The bound $|Z|\leq k$ would be best possible if we require our hitting sets to be radius-$1$ balls. Indeed, Albrechtsen, Huynh, Jacobs, Knappe, and  Wollan~\cite{Albrechtsen2024Menger} constructed a graph $G$ with sets $X,Y\subseteq V(G)$ such that $G$ contains no two anti-complete $(X,Y)$-paths, yet $G-N_G[v]$ contains an $(X,Y)$-path for every $v\in V(G)$. See \cite{hendrey2023induced,gartland2023induced} for further induced Menger results. 

\cref{InducedMenger}, if true, would have immediate application to the study of induced minors. In particular, it would imply an Induced Grid Minor Theorem for $K_{1,t}$-induced-subgraph-free graphs that was conjectured by \citet{DKKMSW2024TI4}; see \cite{HJ2025InducedForest} for further discussion on this connection.

Returning to Geelen's conjecture, we show that the general case can be reduced to the $d=3$ setting. Our proof directly parallels a similar reduction by Seymour and McCarty; see~\cite[Theorem~4]{Albrechtsen2024Menger}.

\begin{obs}
    If \cref{GeelenConj} holds for $d=3$, then it holds for all $d\in \NN$ with $g(d):=d  \cdot g(3)$.
\end{obs}

\begin{proof}
   Let $H$ be the graph obtained from $G$ by adding an edge between every pair of distinct vertices at distance at most $d$ in $G$. If there exists a set $Z\subseteq V(H)$ with $|Z|\leq f(k)$ such that $H-N_H[Z,g(3)]$ has no $\A$-path, then $G-N_G[Z,d\cdot g(3)]$ also has no $\A$-path. Otherwise, $H$ contains $k$ $\A$-paths $P_1',\dots,P_k'$ such that $\dist_H(P_i',P_j')\geq 3$ for all distinct $i,j\in [k]$. Suppose $P_i'=(v_0',v_1',v_2',\dots,v_m')$. For each $j\in [m]$, there is a $(v_{j-1}',v_j')$-path $Q_j$ in $G$ of length at most $d$. Let $P_i$ be a $(v_0',v_m')$-path contained in $G[\bigcup(Q_j\colon j\in [m])]$. Thus $P_1,\dots,P_k$ is a collection of $\A$-paths. For contradiction, suppose $\dist_G(P_i,P_j)<d$ for some $i\neq j$. Then there exists a $(P_i',P_j')$-path in $G$ of length at most $\frac{d}{2}+d+\frac{d}{2}\leq 2d$, implying $\dist_H(P_i',P_j')\leq 2$, a contradiction.
\end{proof}

\textbf{Note:} Albrechtsen, Knappe, and Wollan~\cite{WollanCoarseGraph2024} independently announced a solution to \cref{GeelenConj} for the case $d=2$, with a hitting set consisting of exponentially many balls in $k$, whose radii also depend on $k$, where $k$ is the number of $\A$-paths. Following our work, Distel, Giocanti, Hodor, Legrand-Duchesne, and Micek~\cite{DGHLM2026Gallai} proved a weaker version of \cref{GeelenConj}, using linearly many balls, whose radii depend exponentially on $k$. 

\section{Notation}

See the textbook by \citet{diestel2017graphtheory} for undefined terms and notations. Let $\NN:=\{1,2,\dots\}$. Let $G$ be a graph and $X,Y\subseteq V(G)$. We write $G[X]$ for the subgraph of $G$ induced on $X$. We refer to $X$ and $G[X]$ interchangeably, whenever there is no chance of confusion. We write $G-X$ for the graph $G[V(G)\setminus X]$. An \defn{$(X, Y)$-path} is a path $P=v_1,\ldots,v_k$ with $V(P)\cap X = \{v_1\}$ and $V(P)\cap Y = \{v_k\}$.
The \defn{distance} between $X$ and $Y$, denoted \defn{$\dist_G(X, Y)$} is the length of a shortest $(X, Y)$-path in $G$. For a vertex $v\in V(G)$ and $\ell\in \NN$, let $N_G(v)$ be the set of vertices adjacent to $v$, $N_G[v]=N_G(v)\cup \{v\}$, and $N_G[v,\ell]=\{u\in V(G)\colon \dist_G(u,v)\leq \ell\}$. Then $N_G[v]=N_G[v,1]$.

\section{Proof}

Let $G$ be a graph and $\A\subseteq V(G)$. Our proof of \cref{LongInducedAPaths} employs the well-known frame technique by \citet{Simonovits1967circuits}. An $\A$-frame in $G$ is roughly an induced subcubic tree whose leaves correspond to vertices in $\A$. We construct our $\A$-frame iteratively by taking geodesics from unprocessed $\A$-vertices to the frame while protecting leaves and degree-3 vertices. If we can construct an $\A$-frame with many leaves, we extract from it our collection of anti-complete $\A$-paths. If not, then we can separate the $\A$-frame from the remainder of the graph using a small number of balls. By applying induction on the number of $\A$-paths in the remainder of the graph, we obtain one of our desired outcomes. Note that Ahn, Gollin, and Kwon~\cite{AGHK2025EP} used a similar technique to prove an induced Erd{\H{o}}s-P{\'o}sa theorem.

The following lemma is used for extracting $\A$-paths from the frame.

\begin{lem}[\cite{BHJ2018Frames}]\label{LeafPaths}
    For every $p\in \NN$, every subcubic tree with $p$ leaves contains $\floor{\frac{p}{2}}$ pairwise vertex-disjoint leaf-to-leaf paths, each of length at least $1$.
\end{lem}

Before applying \cref{LeafPaths}, we must overcome two obstacles. First, the leaf-to-leaf paths need not be anti-complete or sufficiently long. Second, the $\A$-frame may not induce a subcubic tree since non-tree edges may be present. To address these difficulties, we introduce the following structure. 

For $\ell\in \NN$, an \defn{$\ell$-hub-tree} is a tuple $(F,T,\A_F,X)$ satisfying the following properties:

\begin{enumerate}[label=(H\arabic*)]
    \item\label{H1} $F$ is a graph;
    \item\label{H2} $T\subseteq F$ is a subcubic tree with $V(T)=V(F)$;
    \item\label{H3} $\A_F$ is the set of degree-$1$ vertices of $T$ which is equal to the set of degree-$1$ vertices of $F$;
    \item\label{H4} $X$ is the set of degree-$3$ vertices in $T$;
    \item\label{H5} for all distinct $x,y\in \A_F$, we have $\dist_F(x,y)\geq \ell$;
    \item\label{H6} for all distinct $x,y\in X$, we have $\dist_F(x,y)\geq 3$; and
    \item\label{H7} for every $uv\in E(F)\setminus E(T)$, there exists $x\in X$ such that $\max\{\dist_T(u,x),\dist_T(v,x)\}\leq 2$.
\end{enumerate}

\begin{lem}\label{AnticompleteHubPaths}
    For all $p,\ell \in \NN$, if $(F,T,\A_F,X)$ is an $\ell$-hub-tree with $|\A_F|=p$, then $F$ contains $\floor{\frac{p}{2}}$ pairwise anti-complete induced $\A_F$-paths, each of length at least $\ell$.
\end{lem}

\begin{proof}
    Let $k:=\floor{\frac{p}{2}}$. By \cref{LeafPaths}, the tree $T$ contains $k$ pairwise vertex-disjoint $\A_F$-paths $P_1',\dots,P_k'$. 
    
    Suppose for contradiction that there exists an edge $uv\in E(F)$ such that $u\in V(P_i')$, $v\in V(P_j')$, and $i\neq j$. Since $P_i'$ and $P_j'$ are vertex-disjoint paths with length at least $1$, both $u$ and $v$ have degree at least $2$ in $F$. Since $P_i'$ and $P_j'$ are $\A_F$-paths, \cref{H3} implies that neither $u$ nor $v$ is an endpoint of its respective path, so both $u$ and $v$ have degree at least $3$ in $F$. If $uv\in E(T)$, then $u$ and $v$ have degree-$3$ in $T$. Then $u,v\in X$ by \cref{H4}, which contradicts \cref{H6}. Thus $uv\in E(F)\setminus E(T)$. By \cref{H7}, there exists $x\in X$ such that $\max\{\dist_T(u,x),\dist_T(v,x)\}\leq 2$. Since $P_i'$ and $P_j'$ are vertex-disjoint, we may assume without loss of generality that $x\not\in V(P_i')$. Let $x'\in V(P_i')$ be the vertex in $P_i'$ that is closest to $x$ in $T$. Since the endpoints of $P_i'$ have degree $1$ in $T$, $x'$ has degree $2$ in $P_i'$. Furthermore, since $x'$ has a neighbour closer to $x$ that is not in $P_i'$, $x'$ has degree $3$ in $T$. But this contradicts \cref{H6} again since $\dist_T(x,x') \leq \dist_T(u,x) \leq 2$. Hence, no such edge $uv$ exists. Thus the paths $P_1',\dots,P_k'$ are pairwise anti-complete in $F$. 
    
    For each $i\in [k]$, let $P_i$ be a vertex-minimal path in $F[V(P_i')]$ whose endpoints coincide with those of $P_i'$. Then $P_1,\dots,P_k$ are also pairwise anti-complete $\A_F$-paths in $F$. By minimality, each $P_i$ is induced in $F$. By \cref{H5}, each $P_i$ has length at least $\ell$. Therefore, $P_1,\dots,P_k$ is our desired collection of $\A_F$-paths.
\end{proof}

We are now ready to prove our main results. \cref{LongInducedAPaths} follows immediately from the next lemma.

\begin{lem}
     For all $k,\ell\in \NN$, every graph $G$ and $\A\subseteq V(G)$ contains one of the following:
    \begin{enumerate}
        \item[(i)] $k$ pairwise anti-complete induced $\A$-paths of length at least $\ell$; or
        \item[(ii)] sets $Z_1,Z_2\subseteq V(G)$ with $|Z_1|\leq (12\max\{\ell,3\}+18)(k-1)$ and $|Z_2|\leq 4(k-1)$ such that $G-(N_G[Z_1]\cap N_G[Z_2,\max\{\ell+1,4\}])$ contains no induced $\A$-path of length at least $\ell$.
    \end{enumerate} 
\end{lem}

\begin{proof}
    Let $\widehat{\ell}:=\max\{\ell,3\}$. We proceed by induction on $k\geq 1$. If $G$ has no induced $\A$-path of length at least $\ell$, then the second outcome occurs for all $k$ by setting $Z_1:=Z_2:=\emptyset$. Thus we may assume that $G$ contains an induced $\A$-path of length at least $\ell$. In which case, for $k=1$, the first outcome occurs. 

    Now assume $k>1$. Before constructing our $\A$-frame, we first handle induced $\A$-paths of intermediate length. Suppose $G$ contains an induced $\A$-path $P$ of length between $\ell$ and $2\ell-1$. Let $G':=G-N_G[P]$. If $G'$ contains $k-1$ pairwise anti-complete induced $\A$-paths $P_1,\dots,P_{k-1}$ of length at least $\ell$, then $P,P_1,\dots,P_{k-1}$ is our desired collection of $\A$-paths. Otherwise, by induction, there exist sets $Z_1',Z_2'\subseteq V(G')$ with $|Z_1'|\leq (12\widehat{\ell}+18)(k-2)$ and $|Z_2'|\leq 4(k-2)$, such that $G'-(N_{G'}[Z_1']\cap N_{G'}[Z_2',\widehat{\ell}+1])$ has no induced $\A$-path of length at least $\ell$. Let $x,y\in V(P)$ be the endpoints of $P$. Setting $Z_1:=Z_1'\cup V(P)$ and $Z_2:=Z_2'\cup \{x,y\}$, we have $|Z_1|\leq (12\widehat{\ell}+18)(k-2)+2\ell<(12\widehat{\ell}+18)(k-1)$, $|Z_2|\leq 4(k-2)+2<4(k-1)$ and $G- (N_{G}[Z_1]\cap N_{G}[Z_2,\widehat{\ell}+1])$ has no induced $\A$-path of length at least $\ell$ (since $N_G[P]\subseteq N_{G}[Z_1]\cap N_{G}[Z_2,\widehat{\ell}+1]$), as required. 
    
    Hence, we may assume that every induced $\A$-path in $G$ of length at least $\ell$ has length at least $2\ell$.

    We define an \defn{$\A$-frame} in $G$ to be a tuple $(F,T,\A_F,X,Y,\widetilde{Y},\overline{\A})$ satisfying:
    
    \begin{enumerate}[label=(A\arabic*)]
        \item\label{A1} $F$ is an induced subgraph of $G$;
        \item\label{A2} $T\subseteq F$ is a subcubic tree with $V(T)=V(F)$;
        \item\label{A3} $\A_F:=\A\cap V(F)$ is the set of degree-$1$ vertices of $T$, which equals the set of degree-$1$ vertices of $F$;
        \item\label{A4} $X$ is the set of degree-$3$ vertices of $T$;
        \item\label{A5} $Y:= \{v\in V(F)\colon \dist_F(v,X\cup \A_F)\leq \widehat{\ell}\}$;
        \item\label{A6} $\widetilde{Y}:=N_G[Y]\setminus V(F)$;
        \item\label{A7} $\overline{\A}:=\A\setminus \A_F$;
        \item\label{A8}  for every $uv\in E(F)\setminus E(T)$, there exists $x\in X$ such that $\max\{\dist_T(u,x),\dist_T(v,x)\}\leq 2$;
        \item\label{A9} for every vertex $v\in V(G)\setminus (V(F)\cup \widetilde{Y})$, if $x,y\in N_G(v)\cap V(F)$ then $\dist_T(x,y)\leq 2$;
        \item\label{A10} no vertex in $\overline{\A}\setminus \widetilde{Y}$ is adjacent to $F$;
        \item\label{A11} $\dist_F(x,y)\geq \ell$ for all distinct $x,y\in \A_F$; and
        \item\label{A12} $\dist_F(x,y)\geq 3$ for all distinct $x,y\in X$.
    \end{enumerate}
    
    We now show that $G$ contains a non-trivial $\A$-frame. 
    
    \begin{claim}\label{ClaimAFrame2}
        $G$ contains an $\A$-frame with $|\A_F|=2$.
    \end{claim}

    \begin{proof}
        Let $F=T=(v_0,v_1,\dots,v_m)$ be a shortest induced $\A$-path in $G$ of length at least $\ell$. By our assumptions, such a path exists and has length at least $2\ell$. Let $\A_F=\{v_0,v_m\}$, $X:=\emptyset$, $Y:= \{v\in V(F)\colon \dist_F(v,\A_F)\leq \widehat{\ell}\}$, $\widetilde{Y}:=N_G[Y]\setminus V(F)$, and $\overline{\A}:=\A\setminus \A_F$.

        We claim that $(F,T,\A_F,X,Y,\widetilde{Y},\overline{\A})$ is an $\A$-frame. By construction, \ref{A1}, \ref{A2}, \ref{A4} to \ref{A8}, \ref{A11} and \ref{A12} are immediate. 
        
        For \cref{A3}, we clearly have $\A_F\subseteq \A\cap V(F)$ with $\A_F$ also being the set of degree-$1$ vertices of both $T$ and $F$. For contradiction, suppose there exists $v\in (\A\cap V(F))\setminus \A_F$. Let $F'$ be a longest subpath of $F$ which has $v$ as an endpoint. Since $F$ has length at least $2\ell$ and $v$ is not an endpoint of $F$, it follows that $F'$ is a strictly shorter induced $\A$-path in $G$ of length at least $\ell$, contradicting our choice of $F$. Thus $\A_F=\A\cap V(F)$.

        For \cref{A9}, suppose there exists $v\in V(G)\setminus (V(F)\cup \widetilde{Y})$ and $x,y\in N_G(v)\cap V(F)$ with $\dist_T(x,y)>2$. Choose such $x,y\in N_G
        (v)\cap V(F)$ to maximise $\dist_T(x,y)$. Let $F'$ be the path obtained from $F$ by replacing the $(x,y)$-subpath of $F$ with the path $xvy$. Then $F'$ is an induced $\A$-path in $G$ whose length is strictly shorter than $F$. Moreover, since $v\not \in \widetilde{Y}$, $F'$ contains the first $\ell+1$ vertices of $F$ and thus has length at least $\ell$, again contradicting our choice of $F$. Thus \cref{A9} holds.

        For \cref{A10}, suppose that some vertex $u\in \overline{\A}\setminus \widetilde{Y}$ has a neighbour in $V(F)$. Choose $v_i\in V(F)$ to be the neighbour of $u$ closest to $v_m$ along $F$. Since $u\not \in \widetilde{Y}$ and so $v_i\not\in Y$, we have $3< i< m-\ell$. Therefore $(u,v_i,v_{i+1},\dots,v_m)$ is an induced $\A$-path in $G$ of length at least $\ell$ that is strictly shorter than $F$, contradicting the choice of $F$. Thus \cref{A10} holds.
    \end{proof}

    Fix an $\A$-frame $(F,T,\A_F,X,Y,\widetilde{Y},\overline{\A})$ in $G$ with $p:=|\A_F|$ maximum. Observe that $(F,T,\A_F,X)$ is an $\ell$-hub tree in $G$. By \cref{AnticompleteHubPaths}, we obtain the following.

    \begin{claim}\label{AFramePaths}
        $F$ contains $\floor{\frac{p}{2}}$ pairwise anti-complete induced $\A_F$-paths of length at least $\ell$.
    \end{claim}

    Since $F$ is an induced subgraph of $G$, we are done if $\floor{\frac{p}{2}}\geq k$. Thus we may assume that $\floor{\frac{p}{2}}< k$. 

    \begin{claim}\label{SizeX}
         $|X|= p-2$. 
    \end{claim}
  
    \begin{proof}
        Since $T$ is a subcubic tree with $p$ leaves and $X$ is the set of degree-$3$ vertices of $T$, we have $|X|=p-2$. 
    \end{proof}
    
    \begin{claim}\label{SizeY}
         $|Y|\leq (4\widehat{\ell}+6)p$. 
    \end{claim}
  
    \begin{proof}
        Let $y\in Y$ and let $P_y=(y=u_0,u_1,\dots,u_b)$ be a shortest $(y,X\cup \A_F)$-path in $F$. By \cref{A5}, we have $b\leq \widehat{\ell}$. If $E(P_y)\subseteq E(T)$, then $\dist_T(y,X\cup \A_F)=\dist_F(y,X\cup \A_F)\leq \widehat{\ell}$. Otherwise, let $i\in \{0,1,\dots,b-1\}$ be minimal such that $u_iu_{i+1}\not\in E(T)$. The path $(u_0,u_1,\dots,u_i)$ lies in $T$, and \cref{A8} gives $\dist_T(u_i,X)\leq 2$. Hence $\dist_T(y,X\cup \A_F)\leq i+2\leq b+1\leq \widehat{\ell}+1$. Thus $Y\subseteq N_{T}[\A_F\cup X,\widehat{\ell}+1]$.
        
        By \cref{A12}, $X$ is an independent set in $T$. Since $X$ is the set of degree-$3$ vertices in $T$, $T-X$ is a linear forest with $2|X|+1=2p-3$ components whose vertices of degree $1$ or $0$ are precisely the vertices in $\A_F\cup N_T(X)$. By considering the expansion within the linear forest, it follows that 
        \begin{align*}
            |Y|&\leq |N_T[\A_F\cup X,\widehat{\ell}+1]|\\
            &\leq |N_T[\A_F\cup X,\widehat{\ell}+1]\setminus (\A_F\cup X)|+|\A_F|+|X|\\
            &\leq 2(2p-3)(\widehat{\ell}+1)+p+(p-2)\\
            &\leq (4\widehat{\ell}+6)p. \qedhere
        \end{align*}
    \end{proof}
        
    Observe that \ref{A5} and \ref{A6} imply the following.
    \begin{claim}\label{Ytilde}
         $\widetilde{Y}\subseteq N_G[Y]\cap N_G[\A_F\cup X,\widehat{\ell}+1]$. 
    \end{claim}
      
    We now show that $\widetilde{Y}$ separates $F$ from $\overline{\A}$.
    
    \begin{claim}\label{ClaimAntiG}
        $G - \widetilde{Y}$ contains no $(\overline{\A},F)$-path.
    \end{claim}
  
    \begin{proof}
        Suppose for contradiction that $G-\widetilde{Y}$ contains an $(\overline{\A},F)$-path. Let $P=(v_0,v_1,\dots, \allowbreak v_{m-1},v_m)$ be a shortest such path, where $v_0\in \overline{\A}$ and $v_m\in V(F)$. Our goal is to append $P$ to the $\A$-frame to construct an $\A$-frame with one additional leaf, contradicting the maximality of $p$. 
        
        By minimality of $P$ and \ref{A6}, the path $P$ has the following properties:
        
        \begin{enumerate}[label=(P\arabic*)]
            \item\label{P1} $\overline{\A}\cap V(P)=\{v_0\}$;
            \item\label{P5} $V(F)\cap V(P)=\{v_m\}$;
            \item\label{P2} $P$ is an induced path;
            \item\label{P3} $P-\{v_{m-1},v_m\}$ is anti-complete to $F$;
            \item\label{P4} $V(P)\cap (Y\cup \widetilde{Y})=\emptyset$;
            \item\label{P6} for every vertex $v\in V(G)\setminus (V(F)\cup \widetilde{Y}\cup V(P))$; if $x,y\in N_G(v)\cap V(P)$, then $\dist_P(x,y)\leq 2$; and
            \item\label{P7} no vertex in $ V(G)\setminus (V(F)\cup \widetilde{Y}\cup V(P))$ has a neighbour in $V(F)$ and a neighbour in $V(P)\setminus\{v_{m-2},v_{m-1},v_m\}$.
        \end{enumerate}

        We now define our extended $\A$-frame. Let $F'$ be the subgraph of $G$ induced by $V(F)\cup V(P)$; $T':=T\cup P$; $\A_F':=\A_F \cup \{v_0\}$; $X':= X \cup \{v_m\}$; $Y':=\{v\in V(T')\colon \dist_{F'}(v,\A_F'\cup X')\leq \widehat{\ell}\}$; 
        $\widetilde{Y}':=N_G[Y']\setminus V(F')$; and $\overline{\A}':=\A\setminus \A_F'$. 
        
        We claim that $(F',T',\A_F',X',Y',\widetilde{Y}',\overline{\A}')$ is an $\A$-frame in $G$. \ref{A1}, \ref{A5}, \ref{A6} and \ref{A7} follow immediately from the definitions. We now verify the remaining properties.
        
        \begin{itemize}
            \item[{\ref{A2}}] By construction, $V(T')=V(F')$. Since $T$ and $P$ are trees, \ref{P5} implies that $T'$ is also a tree. \ref{A5} from the original frame with \ref{P4} implies that $v_m\not \in X\cup \A_F$, and so $\deg_T(v_m)=2$. Since $v_m$ has degree $1$ in $P$, it follows that $\deg_{T'}(v_m)=3$. Thus $T'$ is also subcubic since all other vertices in $T'$ have degree at most $3$ since both $T$ and $P$ are subcubic. 
            
            \item[{\ref{A3}}] By \ref{A3} from the original frame and \ref{P1}, we have $\mathcal{A}_F'=\mathcal{A}_F\cup\{v_0\}=(\mathcal{A}\cap V(F))\cup\{v_0\}=\mathcal{A}\cap V(F')$. By \cref{A10} from the original frame, $v_0$ has no neighbour in $F$ and so $v_0$ is the only vertex in $V(P)$ that has degree $1$ in $T'$ or $F'$. Thus $\A_F'$ is the set of degree-$1$ vertices in both $T'$ and in $F'$. 
            
            \item[{\ref{A4}}] By \ref{A4} of the original frame, the degree-$3$ vertices of $T$ are precisely $X$. As shown above, $v_m$ is the only vertex in $V(P)$ that has degree $3$ in $T'$. Therefore $X'$ is the set of degree-$3$ vertices of $T'$.
            
            \item[{\ref{A8}}] Let $uv\in E(F')\setminus E(T')$. If $uv\in E(F)\setminus E(T)$, then \cref{A8} from the original frame provides $x\in X'$ such that $\max\{\dist_{T'}(u,x),\dist_{T'}(v,x)\}\leq 2$. If $u,v\in V(P)$, then it contradicts \ref{P2} since $uv\notin E(P)$.
            Otherwise, $\{u,v\}\cap (V(P)\setminus V(F)) \neq\emptyset$ and $\{u,v\}\cap (V(F)\setminus V(P))\neq \emptyset$. By \cref{P3}, $\{u,v\}\cap (V(P)\setminus V(F)) = \{v_{m-1}\}$. Since $V(P)\cap \widetilde{Y}=\emptyset$, \cref{A9} applied to the original frame with $v_{m-1}$ being the external vertex gives $\max\{\dist_{T'}(u,v_m),\dist_{T'}(v,v_m)\}\leq 2$. Since $v_m\in X'$, \cref{A8} holds.
           
            \item[{\ref{A9}}] Let $v\in V(G)\setminus (V(F')\cup \widetilde{Y}')$ with $x,y\in N_G(v)\cap V(F')$. If $x,y\in V(F)$, then \cref{A9} from the original frame implies $\dist_{T'}(x,y)\leq 2$. If $x,y\in V(P)$, then \cref{P6} gives $\dist_{T'}(x,y)\leq 2$. So we may assume that $x\in V(P)$ and $y\in V(F)$. By \cref{P7}, $x\in \{v_{m-2},v_{m-1},v_m\}$. In which case, $x\in Y'$ and so $v\in \widetilde{Y}'$, contradicting our choice of $v$. 

            \item[{\ref{A10}}] Suppose that some vertex $u\in \overline{\A}'\setminus \widetilde{Y}'$ has a neighbour in $V(F')$. By \cref{A10} applied to the original frame, $u$ has no neighbour in $V(F)$, and hence it must have a neighbour in $V(P)$. Choose $v_i\in V(P)$ to be the neighbour of $u$ closest to $v_m$ along $P$. Since $u\not \in \widetilde{Y}'$, we have $v_i\not\in Y'$ and so $i>3$. Therefore $(u,v_i,v_{i+1},\dots,v_m)$ is an $(\overline{\A},F)$-path in $G-\widetilde{Y}$ that is strictly shorter than $P$, contradicting the choice of $P$.
        
            \item[{\ref{A11}}] Let $x,y\in \A_F'$ be distinct. If $x,y\in \A_F$, then \cref{A11} from the original frame implies $\dist_{F'}(x,y)\geq \ell$. Otherwise, we may assume that $x\in \A_F$ and $y=v_0$. By \ref{A5} from the original frame and \ref{P4}, every vertex in $F'$ at distance at most $\ell$ from $x$ is contained in $Y\subseteq V(T)$. Since $v_0\not \in V(F)$, it follows that $\dist_{F'}(x,y)> \ell$. 
            
            \item[{\ref{A12}}] Let $x,y\in X'$ be distinct. If $x,y\in X$, then \cref{A12} from the original frame implies $\dist_{F'}(x,y)\geq 3$. Otherwise, we may assume that $x\in X$ and $y=\{v_m\}$. By \ref{A5} from the original frame and \ref{P4}, every vertex in $F'$ at distance at most $3$ from $x$ is contained in $Y$. By \ref{P4} again, we have $\dist_{F'}(x,y)\geq 3$.
        \end{itemize}

        Therefore $(F',T',\A_F',X',Y',\widetilde{Y}',\overline{\A}')$ is an $\A$-frame in $G$ with $|\A_F'|=|\A_F|+1>|\A_F|$, a contradiction.
    \end{proof}

    Let $G'$ be the union of all connected components of $G-\widetilde{Y}$ containing vertices from $\overline{\A}$. Set $k':=k-\floor{\frac{p}{2}}$.

    \begin{claim}
         If $G'$ contains $k'$ pairwise anti-complete induced $\overline{\A}$-path of length at least $\ell$, then $G$ contains $k$ pairwise anti-complete induced $\A$-paths of length at least $\ell$.
    \end{claim}

    \begin{proof}
        By \cref{ClaimAntiG}, no connected component of $G-\widetilde{Y}$ contains a vertex from both $V(F)$ and $\overline{\A}$. Since $\widetilde{Y}\cap V(F)=\emptyset$ (by \ref{A6}), $G'$ is anti-complete to $F$. The claim then follows from \cref{AFramePaths}.
    \end{proof}
    
     Therefore, we may assume that $G'$ does not contain $k'$ pairwise anti-complete induced $\overline{\A}$-path of length at least $\ell$. The next claim completes the proof.

    \begin{claim}
        There exist sets $Z_1,Z_2\subseteq V(G)$ with $|Z_1|\leq (12\widehat{\ell}+18)(k-1)$ and $|Z_2|\leq 4(k-1)$ such that $G-(N_G[Z_1]\cap N_G[Z_2,\widehat{\ell}+1])$ has no induced $\A$-path of length at least $\ell$.
    \end{claim}

    \begin{proof}
        By \cref{ClaimAFrame2} and our choice of $\A$-frame, we have $p\geq 2$, so $k'<k$. By induction, there exist $Z_1',Z_2'\subseteq V(G')$ with $|Z_1'|\leq (12\widehat{\ell}+18)(k'-1)$ and $|Z_2'|\leq 4(k'-1)$ such that $G'-(N_{G'}[Z_1']\cap N_{G'}[Z_2',\widehat{\ell}+1])$ has no induced $\overline{\A}$-path of length at least $\ell$. Set $Z_1:=Z_1'\cup Y$ and $Z_2:=Z_2'\cup \A_F\cup X$. Then, using \Cref{SizeY} we obtain 
        $$|Z_1|\leq (12\widehat{\ell}+18)(k-\floor{\frac{p}{2}}-1)+(4\widehat{\ell}+6)p\leq (12\widehat{\ell}+18)(k-1),$$ 
        $$|Z_2|\leq 4(k-\floor{\frac{p}{2}}-1)+(2p-2)\leq 4(k-1),$$ 
        and $\A_F \subseteq Z_1\cap Z_2$. Thus any $\A$-path in $G-(N_G[Z_1]\cap N_G[Z_2,\widehat{\ell}+1])$ is an $\overline{\A}$-path. Since $G'-(N_{G'}[Z_1']\cap N_{G'}[Z_2',\widehat{\ell}+1])$ has no induced $\overline{\A}$-path of length at least $\ell$, every induced $\overline{\A}$-path in $G-(N_G[Z_1]\cap N_G[Z_2,\widehat{\ell}+1])$ of length at least $\ell$ must contain a vertex from $\widetilde{Y}$. By \cref{Ytilde}, we have $\widetilde{Y}\subseteq N_G[Y]\cap N_G[\A_F\cup X,\widehat{\ell}+1]\subseteq N_G[Z_1]\cap N_G[Z_2,\widehat{\ell}+1]$, completing the proof.
    \end{proof}
    
\end{proof}

\subsection*{Acknowledgements}
While this work was carried out during the second half of 2025, it was motivated by Jim Geelen's conjecture from the 2024 Barbados Graph Theory Workshop held at Bellairs Research Institute in March 2024. We thank him and all the workshop participants for stimulating discussions about coarse graph theory. We also thank Piotr Micek for helpful comments on an earlier version of this paper and the referees for their helpful comments.

{
\fontsize{11pt}{12pt}
\selectfont
	
\hypersetup{linkcolor={red!70!black}}
\setlength{\parskip}{2pt plus 0.3ex minus 0.3ex}

\bibliographystyle{DavidNatbibStyle}
\bibliography{main.bbl}
}

\end{document}